\documentclass[12pt, reqno]{amsart} 
\usepackage[letterpaper, margin = .8in]{geometry} 

\usepackage{lmodern}
\usepackage{amsmath, amsthm, amssymb, amsfonts} 
\allowdisplaybreaks
\usepackage{mathtools}
\usepackage{enumitem}
\usepackage{thmtools}
\usepackage{thm-restate}
\usepackage{graphicx}
\usepackage[dvipsnames]{xcolor}
\usepackage[colorlinks]{hyperref}
\hypersetup{
    linkcolor=NavyBlue,
    citecolor=NavyBlue,
    filecolor=Red,
    urlcolor=NavyBlue,
    menucolor=Red,
    runcolor=Red
}

\newcommand{\Z}{\mathbb{Z}}
\renewcommand{\P}{\mathbb{P}}
\newcommand{\E}{\mathbb{E}}
\newcommand{\ceil}[1]{\left\lceil#1\right\rceil}
\newcommand{\floor}[1]{\left\lfloor#1\right\rfloor}
\renewcommand{\S}[1]{\mathcal S_{2n,#1}}
\newcommand{\Sd}[1]{\mathcal S_{D,#1}}

\newtheorem{thm}{Theorem}[section]
\newtheorem{lem}[thm]{Lemma}
\newtheorem{cor}[thm]{Corollary}

\newtheorem{conj}[thm]{Conjecture}
\theoremstyle{definition}
\newtheorem{defn}[thm]{Definition}

\theoremstyle{remark}

\newtheorem{rmk}{Remark}

\usepackage[backend=biber, doi=false, style=alphabetic, maxalphanames=10, maxnames=10]{biblatex}
\usepackage{comment}
\title{Sum and Difference Sets in Generalized Dihedral Groups}
\author[Ascoli, Cheigh, Dantas~e~Moura, Jeong, Keisling, Lilly, Miller, Ngamlamai, Phang]{Ruben Ascoli, Justin Cheigh, Guilherme~Zeus Dantas~e~Moura, Ryan Jeong, Andrew Keisling, Astrid Lilly, Steven J. Miller, Prakod Ngamlamai, Matthew Phang}

\date{\today}
\thanks{This work was supported by NSF grant DMS1947438, Williams College, and Harvey Mudd College.}

\begin{document}
\maketitle

\begin{abstract}
Given a group $G$, we say that a set $A \subseteq G$ has more sums than differences (MSTD) if $|A+A| > |A-A|$, has more differences than sums (MDTS) if $|A+A| < |A-A|$, or is sum-difference balanced if $|A+A| = |A-A|$. A problem of recent interest has been to understand the frequencies of these type of subsets.

The seventh author and Vissuet studied the problem for arbitrary finite groups $G$ and proved that almost all subsets $A\subseteq G$ are sum-difference balanced as $|G|\to\infty$. For the dihedral group $D_{2n}$, they conjectured that of the remaining sets, most are MSTD, i.e., there are more MSTD sets than MDTS sets. Some progress on this conjecture was made by Haviland et al. in 2020, when they introduced the idea of partitioning the subsets by size: if, for each $m$, there are more MSTD subsets of $D_{2n}$ of size $m$ than MDTS subsets of size $m$, then the conjecture follows.

We extend the conjecture to generalized dihedral groups $D=\Z_2\ltimes G$, where $G$ is an abelian group of size $n$ and the nonidentity element of $\Z_2$ acts by inversion. We make further progress on the conjecture by considering subsets with a fixed number of rotations and reflections. By bounding the expected number of overlapping sums, we show that the collection $\Sd{m}$ of subsets of the generalized dihedral group $D$ of size $m$ has more MSTD sets than MDTS sets when $6\le m\le c_j\sqrt{n}$ for $c_j=1.3229/\sqrt{111+5j}$, where $j$ is the number of elements in $G$ with order at most $2$. We also analyze the expectation for $|A+A|$ and $|A-A|$ for $A\subseteq D_{2n}$, proving an explicit formula for $|A-A|$ when $n$ is prime.
\end{abstract}


\section{Introduction and Main Results}


Given a set of $A$ integers, the sumset and difference set of $A$ are defined as 
\begin{equation}
    A+A\ =\ \{a_1 + a_2: a_1, a_2 \in A\}
    \quad \text{ and } \quad
    A-A\ =\ \{a_1 - a_2: a_1, a_2 \in A\}.
\end{equation}

These elementary operations are fundamental in additive number theory. A natural problem of recent interest has been to understand the relative sizes of the sum and difference sets of sets $A$.

\begin{defn}\label{defn:mstd}
    We say that a set $A$ has \textbf{more sums than differences (MSTD)} if $|A+A| > |A-A|$; has \textbf{more differences than sums (MDTS)} if $|A+A| < |A-A|$; or is \textbf{sum-difference balanced} if $|A+A| = |A-A|$. 
\end{defn}

We intuitively expect most sets to be MDTS since addition is commutative and subtraction is not. Nevertheless, MSTD subsets of integers exist. Nathanson detailed in \cite{nathanson2007problems} the history of the problem, and attributed to John Conway the first recorded example of an MSTD subset of integers, $\{0, 2, 3, 4, 7, 11, 12, 14\}$. Martin and O'Bryant proved in \cite{martin2006many} that the proportion of the $2^n$ subsets $A$ of $\{0, 1, \ldots, n-1\}$ which are MSTD is bounded below by a positive value for all $n\geq 15$. They proved this by controlling the ``fringe'' elements of $A$, those close to $0$ and $n-1$, which have the most influence over whether elements are missing from the sum and difference sets. In \cite{zhao2011characterized}, Zhao gave a deterministic algorithm to compute the limit of the ratio of MSTD subsets of $\{0, 1, \ldots, n-1\}$ as $n$ goes to infinity and found that this ratio is at least $4.28\times 10^{-4}$. For more on the problem of MSTD sets in the integers, see also \cite{hegarty2007explicit} and \cite{nathanson2006sets} for constructive examples of infinite families of MSTD sets, \cite{miller2010explicit} and \cite{zhao2010ballot} for non-constructive proofs of existence of infinite families of MSTD sets, and \cite{hegarty2009almost} and \cite{hogan2013generalized} for an analysis of sets with each integer from $0$ to $n-1$ included with probability $cn^{-\delta}$.

More recently, several authors have examined analogous problems for groups $G$ other than the integers. For example, Do, Kulkarni, Moon, Wellens, Wilcox, and the seventh author studied in \cite{do2015polytopes} the analogous problem for higher-dimensional integer lattices.

For finite groups, although the usual notation for the operation of the group is multiplication, we match the notation from previous work and define, for a subset $A \subseteq G$, its sumset and difference set as
\begin{equation}
    A+A\ =\ \{a_1 a_2: a_1, a_2 \in A\}
    \quad \text{ and } \quad
    A-A\ =\ \{a_1 a_2^{-1}: a_1, a_2 \in A\}.
\end{equation}
Definition~\ref{defn:mstd} of MSTD, MDTS, and sum-difference balanced sets apply in this context.


The approaches used to study MSTD subsets of integers do not generalize for MSTD subsets of finite groups due to the lack of fringes. Zhao proved asymptotics for numbers of MSTD subsets of finite abelian groups as the size of the group goes to infinity in \cite{zhao2010counting}. The seventh author and Vissuet examined the problem for arbitrary finite groups $G$, also with the size of the group going to infinity, and proved Theorem~\ref{thm:almost-all-balanced}.

\begin{thm}[\cite{miller2014most}] \label{thm:almost-all-balanced}
    Let $\{G_n\}$ be a sequence of finite groups, not necessarily abelian, with $|G_n| \to \infty$. Let $A_n$ be a uniformly chosen random subset of $G_n$. Then $\P[A_n + A_n = A_n - A_n = G_n] \to 1$ as $n \to \infty$. In other words, as the size of the finite groups increases  without bound, almost all subsets are balanced (with sumset and difference set equalling the entire group).
\end{thm}


Furthermore, for the case of dihedral groups $D_{2n}$, they proposed Conjecture~\ref{conj:dih-mmstdtmdts}.

\begin{conj}[\cite{miller2014most}] \label{conj:dih-mmstdtmdts}
    Let $n \geq 3$ be an integer. There are more MSTD subsets of $D_{2n}$ than MDTS subsets of $D_{2n}$.  
\end{conj}

Given a set $A\subseteq D_{2n} = \langle r, s \mid r^n, s^2, rsrs\rangle$, define $R$ (resp. $F$) as the set of elements of $A$ of the form $r^i$ (resp. $r^is$), called rotation elements (resp. flip elements). Hence, $A = R \cup F$. Then, we can write
\begin{align}
    A + A\ &=\ (R+R) \cup (F+F) \cup (R+F) \cup (-R+F), \nonumber\\
    A - A\ &=\ (R-R) \cup (F+F) \cup (R+F).
\end{align}

Intuition for Conjecture~\ref{conj:dih-mmstdtmdts} comes from noting that $F+F$ and $R+F$ contribute to both $A+A$ and $A-A$; $R+R$ and $-R+F$ contribute only to $A+A$; and $R-R$ contributes only to $A-A$. 

In 2020, Haviland, Kim, Lâm, Lentfer, Trejos Suáres, and the seventh author made progress towards Conjecture~\ref{conj:dih-mmstdtmdts} by partitioning subsets of $D_{2n}$ by size. They proposed Conjecture~\ref{conj:S2nm-mmstdtmdts} as a means of proving Conjecture~\ref{conj:dih-mmstdtmdts}.

\begin{conj}[\cite{haviland2020more}] \label{conj:S2nm-mmstdtmdts}
    Let $n \geq 3$ be an integer, and let $\mathcal S_{2n, m}$ denote the collection of subsets of $D_{2n}$ of size $m$. For any $m \leq 2n$, $\mathcal{S}_{2n,m}$ has at least as many MSTD sets as MDTS sets.
\end{conj}

They showed that Conjecture~\ref{conj:S2nm-mmstdtmdts} holds for $m = 2$, which we reproduce in this paper, and we also extend their approach to $m = 3$. They also showed that Conjecture~\ref{conj:S2nm-mmstdtmdts} holds for $m > n$ by showing that all sets in $\mathcal{S}_{2n, m}$ are sum-difference balanced. We prove this result in this paper, using Lemma~\ref{>n/2 prelim}. These results are proved in Section \ref{direct analysis}.

\begin{lem}\label{>n/2 prelim}
Let $n\geq 3$ be an integer, and let $A \subseteq D_{2n}$. Let $R$ (resp. $F$) be the subset of rotations (resp. flips) in $A$. Suppose that $|F|>\frac{n}{2}$ or $|R|>\frac{n}{2}$. Then, $A$ cannot be MDTS.
\end{lem}


Furthermore, we extend Conjecture~\ref{conj:dih-mmstdtmdts} as follows. A \textbf{generalized dihedral group} is given by $D = \Z_2\ltimes G$, where $G$ is any abelian group and where the nonidentity element of $\Z_2$ acts on $G$ by inversion. 

\begin{conj}\label{conj:gendih-mmstdtmdts}
    Let $G$ be an abelian group with at least one element of order $3$ or greater, and let $D = \Z_2\ltimes G$ be the corresponding generalized dihedral group. Then, there are more MSTD subsets of $D$ than MDTS subsets of $D$.
\end{conj}

Conjecture~\ref{conj:dih-mmstdtmdts} is a special case of Conjecture~\ref{conj:gendih-mmstdtmdts}, with $G=\Z_n$. We also state Conjecture~\ref{conj:SDm-mmstdtmdts}, analogous to Conjecture~\ref{conj:S2nm-mmstdtmdts}.

\begin{conj}\label{conj:SDm-mmstdtmdts}
    Let $D$ be a generalized dihedral group of size $2n$, and let $\mathcal S_{D,m}$ denote the collection of subsets of $D$ of size $m$. For any $m\leq 2n$, $\mathcal S_{D,m}$ has at least as many MSTD sets as MDTS sets.
\end{conj}
The version of Lemma \ref{>n/2 prelim} that we prove in Section \ref{direct analysis} deals with the generalized dihedral group.

In Section \ref{collision analysis}, we prove our main theorem, verifying Conjecture~\ref{conj:SDm-mmstdtmdts} for the case of $m \leq c_j\sqrt{n}$, where $c_j$ is a constant (independent of $n$) depending only on the quantity $j$, the number of elements of order at most $2$ in the abelian group $G$. More explicitly, we show the following.

\begin{restatable}{thm}{smallm}\label{smallm}
Let $D = \Z_2\ltimes G$ be a generalized dihedral group of size $2n$. Let $\Sd{m}$ denote the collection of subsets of $D$ of size $m$, and let $j$ denote the number of elements in $G$ with order at most $2$. If $6\leq m\leq c_j \sqrt{n}$, where $c_j= 1.3229 / \sqrt{111+5j}$, then there are more MSTD sets than MDTS sets in $\Sd{m}$.

\end{restatable}
See Section \ref{collision analysis} the proof of this result and two related theorems.
We also extend these results to the dihedral group on finitely generated abelian groups in Section \ref{fingenabgroups}.

Next, in Section \ref{expectation section}, we discuss the following result about the expected size of $|A-A|$ when $A$ is a randomly chosen set in $\S{m}$, the collection of subsets of $D_{2n}$ of size $m$.

\begin{restatable}{thm}{expectationdifprime}\label{thm:expectation_dif_prime}
    If $n$ is prime, and $A$ is chosen uniformly at random from $\S{m}$, then
    \begin{equation}
        \E[|A-A|]\ =\ 2n-\frac{nm2^m{\binom{n}{m}}+2n(n-1){\binom{n-m-1}{m-1}}}{m{\binom{2n}{m}}} -\frac{n^2(n-1)}{{\binom{2n}{m}}}\sum_{k=1}^{m-1}\frac{{\binom{n+k-m-1}{m-k-1}}{\binom{n-k-1}{k-1}}}{k(m-k)}.
    \end{equation}
\end{restatable}

Finally, in Section \ref{future work}, we discuss directions for further research.



\section{Direct Analysis}\label{direct analysis}

\subsection{Small Subsets}
For the case of the usual dihedral group $D_{2n}$, we have the following two results.

\begin{restatable}[\cite{haviland2020more}]{lem}{stwontwo}\label{S2n2}
Let $n \geq 3$, and let $\mathcal S_{2n, 2}$ denote the collection of subsets of $D_{2n}$ of size $2$. Then, $\mathcal{S}_{2n,2}$ has strictly more MSTD sets than MDTS sets.
\end{restatable}

\begin{restatable}{lem}{stwonthree}\label{S2n3}
Let $n \geq 3$, and let $\mathcal S_{2n, 3}$ denote the collection of subsets of $D_{2n}$ of size $3$. Then, $\mathcal{S}_{2n,3}$ has strictly more MSTD sets than MDTS sets.
\end{restatable}

The proofs for both these lemmas use basic and somewhat tedious casework; they can be found in Appendix \ref{base cases}. 

Similar results for the generalized dihedral group likely follow from similar arguments. 

\subsection{Large Subsets}\label{m large}
We consider what happens when $m$ gets close to $n$. Here we can prove a result for any generalized dihedral group. For the rest of this section, let $G$ be a finite abelian group of size $n$. Recall that the generalized dihedral group is given by $D = \Z_2\ltimes G$, where the nonidentity element of $\Z_2$ acts on $G$ by inversion. Note that $|D|=2n$. Writing an element of the group $D$ as $(z,g)$ where $z\in \{0,1\}$ and $g\in G$, we write $R_D$ to mean the subset of $D$ consisting of elements with $z=0$ and $F_D$ to mean the subset of $D$ consisting of elements with $z=1$. Note that $|R_D|=|F_D|=n$. For the case where $D = D_{2n} = \Z_2\ltimes \Z_n$ is the usual dihedral group, $R_D$ and $F_D$ are the sets of rotations and flips, respectively; out of convenience, we will use these terms for the general case as well.

It turns out that having $m>n$ ensures that $A$ is balanced.  

\begin{lem}\label{>n/2}
Let $D$ be a generalized dihedral group of size $2n$. Let $A\subseteq D$, and let $R = A\cap R_D$ and $F = A\cap F_D$. If $\max(|R|,|F|) > n/2$, then $R_D \subseteq A+A$ and $R_D \subseteq A-A$.

\end{lem}
\begin{proof}
    Let $L$ be the larger of $R$ and $F$, and define $L_D = R_D$ if $L = R$ and $L_D = F_D$ if $L = F$.

    For each rotation $r \in R_D$, define $rL^{-1} = \{r\ell^{-1} \ |\ \ell \in L\}$ and $rL = \{r\ell \ |\ \ell \in L\}$.
    Note that $|rL^{-1}| = |rL| = |L| > n/2$, and $rL^{-1}$, $rL$, $L$ are subsets of the set $L_D$ which has size $n$.
    Hence, by the inclusion--exclusion principle, $L \cap rL^{-1}$ and $L \cap rL$ are nonempty.
    Thus, $r \in L+L \subseteq A+A$ and $r \in L-L \subseteq A-A$.

    Therefore, as desired, $R_D \subseteq A + A$ and $R_D \subseteq A-A$.
\end{proof}

\begin{rmk}
Note Lemma \ref{>n/2} implies if $\max(|R|,|F|) > n/2$, then $A$ is not MDTS. This follows from the discussion after the statement of Conjecture \ref{conj:dih-mmstdtmdts} (which we explicitly extend to the general dihedral group case in Section \ref{collision analysis}). For a set to be MDTS, $R-R$ must contribute rotations to $A-A$ that the set $A+A$ does not have. But here we have shown that if $\max(|R|,|F|) > n/2$, then $A+A$ has all the rotations.
\end{rmk}

\begin{lem}
Let $D$ be a generalized dihedral group of size $2n$, and let $A\subseteq D$ with $|A| = m$. If $m > n$, then $A + A = A - A = D$.
\end{lem}

\begin{proof}

Let $R$ (resp. $F$) be the subset of rotations (resp. flips) in $A$; hence $A = R \cup F$. Let $L, S$ be the larger and smaller of $R$ and $F$, respectively. Define $|L| = n_1, |S| = n_2$ for $n_1 + n_2 = m > n$. Thus, $n_1 > n/2$.

By Lemma \ref{>n/2}, we have that $R_D \subseteq A+A$ and $R_D \subseteq A-A$. 

For each flip $f \in F_D$, define $fL^{-1} = \{f\ell^{-1} \ |\ \ell \in L\}$ and $fL = \{f\ell \ |\ \ell \in L\}$.

Note that $|fL^{-1}| = |fL| = |L| = n_1$, $|S| = n_2$, and $fL^{-1}, fL, S$ are subsets of $S_D$, which has size $n < n_1 + n_2$. Hence, by the inclusion--exclusion principle, $fL^{-1} \cup S$ and $fL \cup S$ are nonempty.
Thus, $f\in S + L \subseteq A+A$ and $f \in S - L \subseteq A - A$.
Therefore, $F_D \subseteq A+A$ and $F_D \subseteq A-A$. 

Thus, we have $R_D, F_D \subseteq A+A$ and $R_D, F_D \subseteq A-A$, which imply $A + A = A - A = D$.
\end{proof}

\section{Collision Analysis}\label{collision analysis}

Let $G$ be a finite abelian group of size $n$, written multiplicatively. Recall that the generalized dihedral group is given by $D = \Z_2\ltimes G$, where the nonidentity element of $\Z_2$ acts on $G$ by inversion.

This section is dedicated to proving the following.

\smallm*

Note that the theorem is only useful when $c_j\sqrt{n}\geq 6$, or $n\geq (6/c_j)^2$. 

If $n$ is arbitrarily large and $j$ is a constant compared to $n$, we can make a stronger statement: we can replace $c_j$ in the above theorem with a constant arbitrarily close to $\sqrt{2/7} \approx 0.5345$. Specifically, we have the following.

\begin{thm}\label{largen}
For fixed $j$ and $\epsilon>0$, there exists $n_{j,\epsilon}$ with the following property. Let $G$ be an abelian group of size $n \geq n_{j,\epsilon}$ with at most $j$ elements of order $2$ or $1$. Then with $D$ and $\Sd{m}$ defined as in Theorem \ref{smallm}, we have that if $6 \leq m\leq \left(\sqrt{2/7}-\epsilon\right)\sqrt{n}$, then there are more MSTD sets than MDTS sets in $\Sd{m}$.
\end{thm}

We can give a stronger, more general statement on the proportion of MSTD sets in $\Sd{m}$ if $m$ is large and also bounded above by a (smaller) constant times $\sqrt{n}$.

\begin{thm}\label{evensmallerm}
Let $D, j$, and $\Sd{m}$ be defined as in Theorem \ref{smallm}. For any $\epsilon > 0$, there exist $m_\epsilon$ and $c_{\epsilon, j}$ such that if $m_\epsilon\leq m \leq c_{\epsilon,j}\sqrt{n}$, the proportion of MSTD sets in $\Sd{m}$ is at least $1-\epsilon$.
\end{thm}

Here $m_\epsilon$ and $c_{\epsilon,j}$ are independent of $n$, but similarly to before, this theorem is only useful when $n\geq (m_\epsilon/c_{\epsilon,j})^2$.

\begin{rmk}
In practice, Theorems \ref{smallm} and \ref{evensmallerm} are most useful if $j$ is essentially a constant compared to $n$. This is indeed the case for the original dihedral group $D_{2n} = \Z_2\ltimes \Z_n$, where $j=1$ when $n$ is odd and $j=2$ when $n$ is even, yielding $c_j \geq 0.12$. The family of original dihedral groups is also a good example of how to use Theorem \ref{largen}: we can apply that theorem with $j=2$ and $\epsilon$ arbitrarily small to get that for large enough dihedral groups, we can get the coefficient of the $\sqrt{n}$ in the theorem to be very close to $\sqrt{2/7}$, which is a significant improvement over $0.12$. 

However, these theorems are not useful when, for example, $G=\Z_2\times \ldots \times \Z_2 \times \Z_3$. Here $G$ does have an element of order at least $3$, so Conjecture \ref{conj:gendih-mmstdtmdts} applies, but we have $j=n/3$, which is too large for Theorems \ref{smallm} and \ref{evensmallerm} to apply to any values of $m$. In fact, if $j\geq n/100$, then $n\leq (6/c_j)^2$, and Theorem \ref{smallm} does not apply to any values of $m$.
\end{rmk}

We first prove Theorem \ref{smallm} here. Then, in Subection \ref{c0.5} we demonstrate Theorems \ref{largen} and \ref{evensmallerm}.

Recall the definitions of $R_D$ and $F_D$ from Section \ref{m large}. Note that any element in $F_D$ has order $2$ in $D$. Furthermore, any element in $R_D$ has order at most $2$ in $D$ if and only if it has order at most $2$ in $G$.

We begin with a set $A$ with size $m$ and count the number of elements in $A+A$ and $A-A$. In this count, we will make a naive assumption: there are no overlaps between sums and differences that we do not expect to overlap. Decompose $A$ into the union of the set of rotations $R = A\cap R_D$ of and the set of flips $F = A\cap F_D$, and define $k=|F|$. We have:
\begin{align}
    A+A\ &=\ (R+F)\cup(F+R)\cup (R+R)\cup (F+F);\nonumber\\
    A-A\ &=\ (R-F)\cup(F-R)\cup (R-R)\cup (F-F)
\end{align}

Consider first the flips in $A+A$ and $A-A$. In $A+A$, these are in $R+F$ and $F+R$. Note that we do not expect a lot of overlap in general; for a rotation $(0,g_1)\in R$ and a flip $(1,g_2)\in F$, $(0,g_1)\cdot(1,g_2)=(1, g_1g_2)$ does not equal $(1,g_2)\cdot(0,g_1)=(1, g_2g_1^{-1})$
unless $g_1$ has order $1$ or $2$ in $G$. On the other hand, for the flips in $A-A$, we have $R-F = F-R$. This is because $(0,g_1)\cdot (1,g_2)^{-1} = (0,g_1)\cdot (1,g_2) = (1,g_1g_2)$, and $(1,g_2)\cdot(0,g_1)^{-1} = (1,g_2)\cdot(0,g_1^{-1}) = (1,g_1g_2)$, so these two are the same. There are $m-k$ rotations and $k$ flips in $A$, so we thus expect the flips to contribute $2(m-k)k$ to $A+A$ but only $(m-k)k$ to $A-A$. 

Next, consider the rotations in $A+A$ and $A-A$. Begin with $F+F$ and $F-F$. Since all flips have order $2$, these are in fact the same set and thus always contribute equally to $A+A$ and to $A-A$. Also note that if $k\neq 0$ (we will treat the $k=0$ case later), the identity $1$ is contained in $F+F$ and $F-F$.

Next consider $R+R$ and $R-R$. Adding rotations is commutative, so we expect $R+R$ to contribute $\binom{m-k}{2} + (m-k)$ to the size of $A+A$, where the $m-k$ term comes from the sum of each rotation in $R$ with itself. On the other hand, in general $g_1g_2^{-1} \neq g_2g_1^{-1}$, so $R-R$ is expected to contribute $2\binom{m-k}{2}$. Here there is no additional $m-k$ term since when $g_1=g_2$, we have $g_1g_2^{-1}=1$, and $1$ was already counted in $A-A$ from $F-F$.

We now put this all together. For $|A+A|>|A-A|$, we need 
\begin{align}\label{nocol}
&2(m-k)k+\binom{m-k}{2}+(m-k)\ >\ (m-k)k+2\binom{m-k}{2},
\nonumber\\\iff & k\ >\ m/3-1 \text{ and } m\ \neq\ k,
\nonumber\\\iff & m/3\ \leq\ k\ <\ m.
\end{align}
Note that when $k=m$ the set is necessarily balanced as $F+F=F-F$. Further, one can now see why we may assume $k\neq 0$: when $k$ is smaller than $m/3$, we expect the set to be MDTS, and indeed we will assume this is the case.

To use this naive estimate to prove our theorem, we first formalize our assumption that we have minimal overlaps within the sumset.


\begin{defn}
Let $A\subseteq D$ such that $|A|=m$, and let $q = (a,b,c,d)\in A^4$. We say that $q$ represents a \textbf{collision} if $ab = cd$.
\end{defn}

Every collision that occurs in $A$ has the potential to make $|A+A|$ smaller relative to our naive estimate, unless the quadruple $q$ is of a form that we already took into account.  For example, if $q=(a,b,a,b)$, then this is not a collision we need to count as $ab=ab$ trivially. Similarly, collisions of the form $q=(a,b,b,a)$ with $a$ and $b$ both rotations do not subtract from $|A+A|$ in Equation \eqref{nocol} as we already accounted for commutativity of addition for rotations. And finally, if $q=(a,b,c,d)$ is a collision where $a$, $b$, $c$, and $d$ are all flips, then $q$ does not impact Equation \eqref{nocol} as $F+F=F-F$ do not affect the relative sizes of the sum and difference sets. We refer to these three kinds of quadruples as \textbf{redundant}.

Every non-redundant collision $(a,b,c,d)$ of $A$, together with $(c,d,a,b)$, decreases the size of $A+A$ by at most $1$ from our naive estimate. Let $X_A$ be half the total number of non-redundant collisions of $A$. Then combining the above analysis with Equation \eqref{nocol}, we are guaranteed to have that $A$ is MSTD when

\begin{align}
&2(m-k)k+\binom{m-k}{2}+(m-k) - X_A \ >\  (m-k)k+2\binom{m-k}{2},
\nonumber\\\iff& m/3 + \frac{2X_A}{3(m-k)}\ \leq\ k<m,
\nonumber\\\iff&3k^2-4mk+m^2+2X_A\ \leq\ 0 \text{ (and $k\ <\ m$)}. \label{withcol}
\end{align}
We use the quadratic equation to solve for when the above quantity equals $0$ and obtain $k=(1/6)(4m\pm \sqrt{16m^2-12(m^2+2X_A)}) = (1/3)(2m\pm \sqrt{m^2-6X_A})$. Thus Equation \eqref{withcol} is satisfied when:
\begin{align}\label{withcol2}
    \frac{2m- \sqrt{m^2-6X_A}}{3}\ \leq\ k\ \leq\ \frac{2m+ \sqrt{m^2-6X_A}}{3} \ \ \text{ (and $k\ <\ m$)}.
\end{align}


\begin{rmk} We are assuming that no ``collisions" of the form $ab^{-1}=cd^{-1}$ happen to lower the size of $A-A$. Because our objective is to guarantee $|A+A|>|A-A|$ for a large proportion of $A$, this assumption still gives a sufficient condition on $X_A$ and $k$.\end{rmk}

One therefore sees that for most values of $k$, when the number of collisions is not too large, $A$ is MSTD. More formally, suppose that $A$ is chosen randomly out of the subsets of $D$ with size $m$. Suppose that the expectation value of $X_A$ is bounded above by $c_1m^2$, where $c_1= 7/1152 \approx 0.006076$. Then the actual value of $X_A$ exceeds $12c_1m^2$ at most $1/12$ of the time by Markov's inequality. Of sets $A$ with $5m/12\leq k\leq 11m/12$, the actual value of $X_A$ exceeds $12c_1m^2$ at most $2/12=1/6$ of the time. Thus when $5m/12\leq k\leq 11m/12$, Equation \eqref{withcol2} is true at least $5/6$ of the time since when $X_A= 12c_1m^2$, the equation reads
\begin{align}\label{withcol3}
    &\frac{2m- \sqrt{m^2-72c_1m^2}}{3}\ \leq\ k\ \leq\ \frac{2m+ \sqrt{m^2-72c_1m^2}}{3},
    \nonumber\\\iff& \frac{2m- \sqrt{9/16\cdot m^2}}{3}\ \leq\  k\ \leq\ \frac{2m+ \sqrt{9/16\cdot m^2}}{3},
    \nonumber\\\iff& \frac{5m}{12}\ \leq\ k\ \leq\ \frac{11m}{12}.
\end{align}

Now, we need to make sure that for our values of $m$, more than $(1/2)/(5/6)=3/5$ proportion of sets in $\Sd{m}$ have $5m/12\leq k\leq 11m/12$. This will ensure that a proportion greater than $1/2$ of sets in $\Sd{m}$ satisfy Equation \eqref{withcol2} and are therefore MSTD.

More formally speaking, we require
\begin{align}\label{5m/12}
&\lim_{n\to\infty} \frac{\sum_{k=\ceil{5m/12}}^{\floor{11m/12}} \binom{n}{k}\binom{n}{m-k}}{\binom{2n}{m}}\ >\ 3/5\\
\iff&\frac{m!}{2^m}\sum_{k=\ceil{5m/12}}^{\floor{11m/12}}\frac{1}{k!(m-k)!}\ >\ 3/5, \label{5m/12simplified}
\end{align}
where we may take the limit as $n\to\infty$ because the left hand side of Equation \eqref{5m/12} decreases with increasing $n$. Equation \eqref{5m/12simplified} can be verified numerically to be true for $m\geq 6$.\footnote{In fact, when $m$ is large, we expect almost all of the sets in $\Sd{m}$ to have $5m/12\leq k \leq 11m/12$; see Subection \ref{c0.5} for further discussion on this fact. For smaller values of $m\geq 6$, one can verify Equations \eqref{5m/12} and \eqref{5m/12simplified} using the following Desmos link: \url{https://www.desmos.com/calculator/e4zqwbmmcr}.}

The problem has therefore been reduced to placing an upper bound on the values of $m$ such that the expected value of $X_A$ is at most $c_1m^2$ when $A$ is chosen uniformly at random from $\Sd{m}$. The following lemma gives us the result we need.

\begin{lem}\label{1overn}
When $A$ is chosen uniformly at random from $\Sd{m}$, we have 
\begin{equation}
    \E[X_{A}]\ \leq\  \left(\frac{7}{32}+\frac{1}{m}+\frac{5j}{8m^2}\right)\frac{m^4}{n}.
\end{equation}
\end{lem}

Much of the machinery of this proof lies in Lemma \ref{1overn}; its proof is rather technical and can be found in Subsection \ref{technical lemmas}.

When $m\geq 6$, Lemma \ref{1overn} implies that under the hypothesis of the lemma, \begin{equation}\label{c2overn} \E[X_{A}]\ \leq\ c_2\frac{m^4}{n}, \quad \text{ where } c_2\  =\ \frac{7}{32}+\frac{1}{6}+\frac{5j}{288}\ =\ \frac{111+5j}{288}.\end{equation}

We are ready to complete the proof of Theorem \ref{smallm}. Recall that we wanted to have $\E[X_A] \leq c_1m^2$ to ensure most subsets of size $m$ would be MSTD. Thus the requisite upper bound on $m$ is determined as

\begin{gather}
    c_2\frac{m^4}{n}\ \leq\ c_1 m^2, \nonumber\\
    \iff m\ \leq\  \sqrt{\frac{c_1 n}{c_2}}\ =\  c_j\sqrt{n},
\end{gather}
where $c_j= \sqrt{c_1/c_2} =\sqrt{7/(4(111+5j))}\approx 1.3229/\sqrt{5j+111}$. 

This concludes the proof of Theorem \ref{smallm}.

\subsection{Proof of Theorems \ref{largen} and \ref{evensmallerm}}\label{c0.5}

In this section we prove Theorem \ref{largen}, and in the process we outline the steps needed to prove Theorem \ref{evensmallerm}. We now assume that $n>n_{j,\epsilon}$, where $j$ is a constant upper bound on the number of elements of order at most 2 in the group $G$ and $n_{j,\epsilon}$ is sufficiently large.

Having proved Theorem \ref{smallm} for $6\leq m\leq (1.3229/\sqrt{5j+111})\sqrt{n}$, we may now assume $m\geq (1.3229/\sqrt{5j+111})\sqrt{n}$; that is, $m$ is now large. We will follow similar steps to the previous proof, but using this assumption, we will increase $c_1$ to be arbitrarily close to $1/16$, and we will decrease $c_2$ to be arbitrarily close to $7/32$. Then we will have that the coefficient of the $\sqrt{n}$, which is $\sqrt{c_1/c_2}$ (as discussed in the previous proof), is arbitrarily close to $\sqrt{2/7}$.

Take small $\epsilon_1>0$. Note that inside of $\Sd{m}$, the distribution of values of $k$ is a hypergeometric distribution. This is because one can construct a random set in $\Sd{m}$ by taking $m$ random elements of the group $D$ without replacement, one at a time; to begin with there is a $1/2$ chance each time that we choose a flip. Thus since $n$ is very large and $j$ is fixed, having $m\geq (1.3229/\sqrt{5j+111})\sqrt{n}$ is sufficient for a proportion at least $1-\epsilon_1$ of sets in $\Sd{m}$ to have $k\in[(1/2-\epsilon_1)m, (1/2+\epsilon_1)m]$.

Going back to Equation \eqref{withcol2}, we thus see that we just need $\sqrt{m^2-6X_A}$ to be at least $m/2+\epsilon_1$, or $6X_A \leq m^2-(m/2+\epsilon_1)^2$, slightly more than half the time when $k$ is in the relevant interval. More specifically, we need \begin{equation}\label{xa stuff} \P\left[6X_A \leq \frac{3m^2}{4}-m\epsilon_1-\epsilon_1^2\ \Big|\  (1/2-\epsilon_1)m\leq k\leq (5/6+\epsilon_1)m \right] \ \geq\ \frac{1}{2}\cdot\frac{1}{(1-\epsilon_1)^2}.\end{equation}

Then, among sets with $(1/2-\epsilon_1)m\leq k\leq (5/6+\epsilon_1)m$ (which, recall, form a proportion of at least $1-\epsilon_1$ of sets in $\Sd{m}$), at least a proportion of $(1/2)/(1-\epsilon_1)^2$ satisfy Equation \eqref{withcol2}. This means that a proportion of at least $(1/2)/(1-\epsilon_1) > 1/2$ of sets in $\Sd{m}$ are MSTD.

For Equation \eqref{xa stuff} to hold, we claim that it suffices to have the following probability bound, not conditioned on the size of $k$:
\begin{equation}\label{xa stuff 2}
\P\left[X_A \leq \frac{m^2}{8} - \frac{m\epsilon_1}{6}-\frac{\epsilon_1^2}{6}\right]\ \geq\  \frac{1}{2}\cdot \frac{1}{1-\epsilon_1}+\epsilon_1.
\end{equation}
To see why Equation \eqref{xa stuff 2} implies Equation \eqref{xa stuff}, call $B$ the event that $(1/2-\epsilon_1)m\leq k\leq (5/6+\epsilon_1)m$ and $C$ the event that $X_A \leq \frac{m^2}{8} - \frac{m\epsilon_1}{6}-\frac{\epsilon_1^2}{6}$.  Then, we manipulate conditional probabilities as follows.
\begin{gather}\label{xa stuff 3}
\P[C]\ =\ \P\left[C\ |\ B\right]\P[B] + \P\left[C\ |\ \neg {B}\right]\P[\neg {B}]\ \leq\ \P\left[C\ |\ B\right]\P[B]+\P[\neg B],\nonumber\\
\iff\P\left[C\ |\ B\right]\ \geq\  \frac{\P[C]-\P[\neg B]}{\P[B]}\ =\ \frac{\P[C]-(1-\P[B])}{\P[B]}\  =\ 1-\frac{1-\P[C]}{\P[B]}.
\end{gather}
Since $\P[B]\geq 1-\epsilon_1$ and Equation \eqref{xa stuff 2} says that $\P[C]\geq (1/2)/(1-\epsilon_1)+\epsilon_1$, we have that if Equation \eqref{xa stuff 2} is true, then
\begin{align}
    \P[C\ |\ B]\ \geq\  1-\frac{1-((1/2)/(1-\epsilon_1)+\epsilon_1)}{1-\epsilon_1}\ =\  1-\frac{(1-\epsilon_1) - (1/2)/(1-\epsilon_1)}{1-\epsilon_1}\ =\  \frac{1/2}{(1-\epsilon_1)^2},
\end{align}
and the claim is shown.

To ensure that Equation \eqref{xa stuff 2} is true, we require \begin{equation}\label{exa} \E[X_A]\ \leq\ \left(1-\epsilon_1 - \frac{1/2}{1-\epsilon_1}\right)\left(\frac{m^2}{8}-\frac{m\epsilon_1}{6}-\frac{\epsilon_1^2}{6}\right).
\end{equation}
Then by Markov's inequality, the probability that $X_A$ exceeds $\frac{m^2}{8}-\frac{m\epsilon_1}{6}-\frac{\epsilon_1^2}{6}$ is at most $1-\epsilon_1 - \frac{1/2}{1-\epsilon_1}$, which is equivalent to Equation \eqref{xa stuff 2}. 

We may now choose a small value $\epsilon_2$ such that Equation \eqref{exa} is true if 
\begin{equation}\label{exa2}
    \E[X_A]\ \leq\  \left(\frac{1}{16}-\epsilon_2\right)m^2.
\end{equation}
Notice that in the limit $\epsilon_1\to 0$, Equation \eqref{exa} boils down to the statement that $\E[X_A] \leq m^2/16$, so we can make $\epsilon_2$ be as small as desired by making $\epsilon_1$ be small.

We now revisit Lemma \ref{1overn}. We are now assuming that $m$ is large and $j$ is a constant, so only the first term dominates:
\begin{equation}\label{exai}
    \E[X_A]\ \leq\  \left(\frac{7}{32}+\epsilon_3\right)\frac{m^4}{n}.
\end{equation}
We wanted to have
$\E[X_A]\leq (1/16-\epsilon_2)m^2$. Thus the upper bound on $m$ is now determined as 
\begin{gather}
\left(\frac{7}{32}+\epsilon_3\right)\frac{m^4}{n} \ \leq\  \left(\frac{1}{16}-\epsilon_2\right)m^2 \nonumber\\ 
    \iff m\ \leq\ c\sqrt{n},
\end{gather}
where \begin{align}
    c\ =\ \sqrt{\frac{1/16-\epsilon'}{7/32+\epsilon_3}}.
\end{align}
If $n$ is arbitrarily large and $j$ is constant compared to $n$, we can choose $\epsilon_2$ and $\epsilon_3$ to be very small, so that $c$ is arbitrarily close to $\sqrt{(1/16) / (7/32)} \approx 0.5345$.

This completes the proof of Theorem \ref{largen}. To prove Theorem \ref{evensmallerm}, we follow a very similar method. We choose $m_\epsilon$ large enough that almost all of the sets in $\Sd{m}$ have $(0.5-\epsilon_1)m\leq k\leq (0.5+\epsilon_1)m$. The only difference is that now, in Equation \eqref{xa stuff} we replace the right-hand side with $(1-\epsilon_1)$, so that the proportion of MSTD sets is at least $(1-\epsilon_1)^2$. (We may choose $\epsilon_1$ so that $(1-\epsilon_1)^2$ equals the desired $(1-\epsilon)$). Then in Equation \eqref{xa stuff 2} we replace the right-hand side with $1-(1-\epsilon_1)\epsilon_1$, leading to the analog of Equation \eqref{exa2} being that $\E[X_A]\leq \epsilon_2 m^2$ for some small $\epsilon_2$ depending on $\epsilon_1$. The rest of the proof continues as before, leading to a coefficient $c_{\epsilon,j}$ proportional to $\sqrt{\epsilon_2}$. This completes the proof.
    
\subsection{Proof of Lemma \ref{1overn}}\label{technical lemmas}
We now prove Lemma \ref{1overn}. Recall that $X_A$ is defined to be half the number of non-redundant collisions in the set $A$, and we are interested in bounding above the expectation value of $X_A$ when $A$ is chosen uniformly at random from $\Sd{m}$. 

To more easily count the collisions in $A$, we make the following definition.
\begin{defn} A \textbf{redundant triple} is a triple $(a,b,c)\in D^3$ such that the quadruple $(a,b,c,c^{-1}ab)$ is redundant. That is, a triple $(a,b,c)$ is redundant if $a=c$, or if $a, b,$ and $c$ are all flips, or if $b=c$ and $a$ and $b$ are both rotations. Denote $T\subseteq D^3$ to be the set of non-redundant triples.
\end{defn}

Define the function $\chi: \Sd{m} \times T \to \{0,1\}$ by $\chi(A,t) = 1$ if for the non-redundant triple $t=(a,b,c)$, the element $c^{-1}ab$ is in $A$, and $\chi(A,t)=0$ otherwise.

For a fixed set $A$, the set $A^3\cap T$ is the set of non-redundant triples with all three elements contained in $A$. Notice that we have 
\begin{equation}
X_A\ =\ \frac{1}{2}\sum_{t\in A^3\cap T}\chi(A,t).\end{equation}
That is, the number of non-redundant collisions in $A$ is the same as the number of non-redundant triples $(a,b,c)\in A^3\cap T$ such that the element $d=c^{-1}ab$ is in $A$, forming a quadruple $(a,b,c,d)$ representing a collision $ab=cd$.

By definition of expectation value, we write 
\begin{equation}
    \E[X_A]\ =\ \sum_{A\in \Sd{m}} \P[A] X_A.
\end{equation}
Since $A$ is chosen uniformly at random from the $\binom{2n}{m}$ sets in $\Sd{m}$, we have $\P[A] = 1/\binom{2n}{m}$. Thus,
\begin{equation}
    \E[X_A]\ =\ \frac{1}{\binom{2n}{m}} \sum_{A\in \Sd{m}} \frac{1}{2}\sum_{t\in A^3\cap T}\chi(A,t).
\end{equation}
We swap the order of the sums.
\begin{equation}
\E[X_A]\ =\ \frac{1}{2}\frac{1}{\binom{2n}{m}}\sum_{t\in T}\sum_{\substack{A\in\Sd{m} \\ A^3\ni t}} \chi(A,t).
\end{equation}
To compute the inner sum, we must simply count the number of sets $A\in \Sd{m}$ with $t=(a,b,c)\in A^3$ such that $A$ contains $c^{-1}ab$. That is,
\begin{equation}\label{sumtriples}
    \E[X_A]\ =\  \frac{1}{2}\frac{1}{\binom{2n}{m}}\sum_{(a,b,c)\in T} |\{A\in \Sd{m}\ |\ a,b,c,c^{-1}ab \in A\}|.
\end{equation}
We now break this sum into seven pieces for different kinds of triples $(a,b,c)\in T$. These are:
\begin{itemize}
    \item $T_1\ =\ \{(a,b,c)\in T\ |\  a,b,c,c^{-1}ab \text{ are distinct }\}$
    \item $T_2\ =\ \{(a,b,c)\in T\ |\  a=b;\  a, c,c^{-1}a^2 \text{ are distinct}\}$
    \item $T_3\ =\ \{(a,b,c)\in T\ |\  b=c;\  a, c,c^{-1}ac \text{ are distinct}\}$
    \item $T_4\ =\ \{(a,b,c)\in T\ |\  c^{-1}ab=a;\  a, b, c \text{ are distinct}\}$
    \item $T_5\ =\ \{(a,b,c)\in T\ |\  c^{-1}ab=c;\  a, b, c \text{ are distinct}\}$
    \item $T_6\ =\ \{(a,b,c)\in T\ |\  b=c; a=c^{-1}ac; \  a, c \text{ are distinct}\}$
    \item $T_7\ =\ \{(a,b,c)\in T\ |\  a=b; c=c^{-1}a^2; \  a, c \text{ are distinct}\}$
\end{itemize}
We have $T = \bigcup_{i=1}^7 T_i$, for these seven cases cover all the cases of possible equalities between the four elements except for those where $a=c$, or equivalently, $c^{-1}ab=b$, since those cases are redundant triples. Furthermore, this union is disjoint.

Note that for triples $(a,b,c)\in T_1$, the quantity $|\{A\in \Sd{m}\ |\ a,b,c,c^{-1}ab \in A\}|$ is given by $\binom{2n-4}{m-4}$ since we are requiring four distinct elements to be in $A$, and we have $2n-4$ choices for the remaining $m-4$ elements. For triples in $T_2, T_3, T_4,$ and $T_5$, we are requiring three distinct elements to be in $A$, so we have $|\{A\in \Sd{m}\ |\ a,b,c,c^{-1}ab \in A\}| = \binom{2n-3}{m-3}$. Finally, for triples in $T_6$ and $T_7$ we have $|\{A\in \Sd{m}\ |\ a,b,c,c^{-1}ab \in A\}| = \binom{2n-2}{m-2}$. 

Therefore, from Equation \eqref{sumtriples}, we may write:
\begin{align}
    \E[X_A]\ &=\  \frac{1}{2}\frac{1}{\binom{2n}{m}}\left[\binom{2n-4}{m-4} |T_1| + \binom{2n-3}{m-3}\left(|T_2|+|T_3|+|T_4|+|T_5|\right)+\binom{2n-2}{m-2}\left(|T_6|+|T_7|\right)\right]\nonumber\\
    &\leq\ \frac{1}{2}\left[\left(\frac{m}{2n}\right)^4 |T_1| + \left(\frac{m}{2n}\right)^3\left(|T_2|+|T_3|+|T_4|+|T_5|\right)+\left(\frac{m}{2n}\right)^2\left(|T_6|+|T_7|\right)\right],\label{exas}
\end{align}
where in the second line we used the fact that
\begin{equation}
    \frac{\binom{2n-4}{m-4}}{\binom{2n}{m}}\ =\  \frac{m(m-1)(m-2)(m-3)}{2n(2n-1)(2n-2)(2n-3)}\ \leq\ \left(\frac{m}{2n}\right)^4,
\end{equation} and similarly for the other two terms.

Now, to use Equation \eqref{exas} to find an upper bound on $\E[X_A]$, all that remains to be done is find an upper bound on each of the $|T_i|$'s. We do so next.

We bound $|T_1|$ using the trivial inequality $|T_1|\leq |T|$. There are $(2n)^3$ total triples in $D^3$, but we may subtract the redundant triples, including the $n^3$ triples consisting of three flips. Thus we obtain \begin{equation} |T_1|\ \leq\ 7n^3.\end{equation}

We bound $|T_2|$ and $|T_3|$ next. We have $2n$ choices for $b$. In $T_2$, $a$ must equal $b$, and in $T_3$, $c$ must equal $b$, and in both, we have at most $2n$ choices for the remaining value of $a$ or $c$. The extra condition that $c^{-1}ab$ is distinct from the others only lowers $|T_2|$ and $|T_3|$, so we do not have to take it into account to obtain an upper bound. Thus, \begin{equation}
    |T_2|,\ |T_3|\ \leq\ 4n^2.
\end{equation}

For $|T_4|$ and $|T_5|$, we have $2n$ choices for $a$ and $2n$ choices for $c$, but $b$ must be the element $a^{-1}ca$ for $T_4$ or $a^{-1}c^2$ for $T_5$. Thus we have
\begin{equation}
    |T_4|,\ |T_5|\ \leq\ 4n^2.
\end{equation}

When considering $T_6$, we note that since $b=c$, the triple $(a,b,c)$ is redundant if $a$ and $c$ are both rotations, or if both are flips. So, one must be a rotation and the other must be a flip, and we require that $a=c^{-1}ac$, or $ca=ac$. Thus to bound $|T_6|$ we must count the number of pairs of elements with $(0, g_1) \cdot (1, g_2) = (1,g_2)\cdot (0,g_1)$, that is, $(1,g_1g_2) = (1,g_2g_1^{-1})$. This happens if and only if $g_1^{-1}=g_1$, or $g_1^2=1$. Recalling that there are $j$ elements in $G$ with order $2$ or less, there are therefore $j$ choices for the rotation element, and $n$ choices for the flip element. We multiply by $2$ since $a$ can be either the flip or the rotation, and $c$ is the other of the two. Thus,
\begin{equation}
    |T_6|\ \leq\ 2nj.
\end{equation}
Finally, for $T_7$, we again split into two cases: firstly where $a$ is a rotation, and secondly where $a$ is a flip, so $c$ is a rotation to avoid redundancy. Here we require $c=c^{-1}a^2$, or $c^2=a^2$. For the first case, we first consider the number of pairs with $a^2=c^2=1$. Since $a$ must be a rotation, there are only $j$ choices for $a^2=1$; $c$ can be a flip or a rotation, so there are $n+j$ ways to have $c^2=1$. So, there are at most $(n+j)j$ pairs of this kind. If $a$ is a rotation and $a^2\neq 1$, then $c$ must also be a rotation since otherwise $c^2=1\neq a^2$. Thus $a$ and $c$ commute, so $a^2=c^2\iff (ac^{-1})^2 = 1$. There are $n-j$ choices of $a$ with $a^2\neq 1$, and for each, there are $j$ choices of $c$ which have $(ac^{-1})^2=1$. Thus there are at most $(n-j)j$ pairs of this kind.

Next we consider the case where $a$ is a flip and $c$ is a rotation. Here $a^2=1$, so there are only $j$ choices for $c$ that have $c^2=a^2$. Thus there are $nj$ pairs of this kind. In total,
\begin{equation}
    |T_7|\ \leq\ (n+j)j+(n-j)j+nj\ =\ 3nj.
\end{equation}
We now substitute these bounds on each $|T_i|$ into Equation \eqref{exas}.
\begin{align}
\E[X_A]\ &\leq\  \frac{1}{2}\left[\left(\frac{m}{2n}\right)^4 \left(7n^3\right) + \left(\frac{m}{2n}\right)^3\left(4n^2+4n^2+4n^2+4n^2\right)+\left(\frac{m}{2n}\right)^2\left(2nj+3nj\right)\right]\nonumber\\
&=\ \left(\frac{7}{32}+\frac{1}{m}+\frac{5j}{8m^2}\right)\frac{m^4}{n}.
\end{align}
This concludes the proof of Lemma \ref{1overn}.

\subsection{Finitely Generated Abelian Groups}\label{fingenabgroups}
We transition to a discussion of finitely generated abelian groups $G$. When $|G| < \infty$ Theorems \ref{smallm}, \ref{largen}, and \ref{evensmallerm} hold, so the remaining case is when $G$ is an infinite group. However, we must make some restriction as to ensure taking subsets uniformly at random is well defined. By the fundamental theorem of finitely generated abelian groups,
\begin{equation}
    G \cong \mathbb{Z}^{r_0} \oplus \mathbb{Z}_{q_1}^{r_1} \oplus \cdots \oplus \mathbb{Z}_{q_k}^{r_k},
\end{equation}
where $q_i$ are powers of (not necessarily distinct) primes. We denote elements of $G$ as a tuple $(g_{0,1},g_{0,2},\dots,g_{0,r_0},g_{1,1},\dots,g_{k,r_k}) \in G$,
where $g_{0,b} \in \mathbb{Z}$ and $g_{a,b} \in \mathbb{Z}_{q_a}$ for $a > 0$. Since this section will occasionally require us to deal with multiple groups simultaneously, we update our notation of $D$ as the generalized dihedral group of $G$ to the standard $\mathrm{Dih}(G)$. We still denote elements of $\mathrm{Dih}(G)$ as $(z,g)$, where $z \in \{0,1\}, g \in G$. 

For some fixed $\alpha \in \mathbb{N}$, we will consider taking subsets uniformly at random from the finite 
\begin{equation}
    \mathrm{Dih}(G_\alpha) = \{(z,(g_{0,1},\dots,g_{0,r_0},\dots,g_{k,r_k})) \in \mathrm{Dih}(G) \ |\ 0 \leq g_{0,b} < \alpha\}.
\end{equation}
Our goal is to leverage Theorems \ref{smallm}, \ref{largen}, and \ref{evensmallerm}, and, to do so, we will refer to $\mathrm{Dih}(G')$, where \begin{equation}
G' = \mathbb{Z}_\alpha^{r_0} \oplus \mathbb{Z}_{q_1}^{r_1} \oplus \cdots \oplus \mathbb{Z}_{q_k}^{r_k}.
\end{equation}
To adhere to prior notation, let $j$ be the number of elements in $G'$ that are at most order $2$ and let $\mathcal{S}_m$ denote the collection of subsets of $\mathrm{Dih}(G_\alpha)$ that have size $m$. Then we get the following corollaries of Theorems \ref{smallm}, \ref{largen}, and \ref{evensmallerm}:
\begin{cor}\label{cor_small_m}
If $6 \leq m \leq c_j\sqrt{n}$, then there are more MSTD than MDTS in $\mathcal{S}_m$.
\end{cor}
\begin{cor}\label{cor_large_n}
If $n\geq n_{j,\epsilon}$, then if $6\leq m\leq \left(\sqrt{2/7}-\epsilon\right)\sqrt n$, then there are more MSTD than MDTS sets in $\mathcal S_m$.
\end{cor}
\begin{cor}\label{cor_smaller_m}
For any $\epsilon > 0$, there exist $m_\epsilon$ and $c_{\epsilon, j}$ such that if $m_\epsilon\leq m \leq c_{\epsilon,j}\sqrt{n}$, the proportion of MSTD sets in $\mathcal{S}_m$ is at least $1-\epsilon$.
\end{cor}

\begin{proof}
We will establish a bijection $\phi: \mathrm{Dih}(G_\alpha) \rightarrow \mathrm{Dih}(G')$
such that, if $(a,b,c,d) \in (\mathrm{Dih}(G_\alpha))^4$ is a collision, then $(\phi(a),\phi(b),\phi(c),\phi(d)) \in (\mathrm{Dih}(G'))^4$ is also a collision. Since we are still working with generalized dihedral groups, this immediately implies Corollaries \ref{cor_small_m}, \ref{cor_large_n}, and \ref{cor_smaller_m}, as the number of non-degenerate collisions in $\mathrm{Dih}(G_\alpha)$ is bounded above by the number of non-degenerate collisions in $\mathrm{Dih}(G')$.

Let $\phi: \mathrm{Dih}(G_\alpha) \rightarrow \mathrm{Dih}(G')$ defined by $\phi((z,(g_{0,1},\dots,g_{k,r_k}))) = (z,(g_{0,1},\dots,g_{k,r_k})).$ As defined, this is clearly a bijection. Let $(x_1,x_2,x_3,x_4) \in (\mathrm{Dih}(G_\alpha))^4$ be a collision. By definition, $x_1x_2 = x_3x_4$. Let $x_j = (z_j, (g_{0,1,j},\dots,g_{k,r_k,j})$. Define the binary operation $\star_t: \Z_t \times \Z_t \rightarrow \Z_t$ by
\begin{equation}
    g_1 \star g_2 = \begin{cases}
        g_1 + g_2 \mod t, \ z_1 = 0, \\ 
        g_1 - g_2 \mod t, \ z_1 = 1.
                    \end{cases}
\end{equation}      
Then we get
\begin{equation}
x_1x_2 = (z_1 + z_2 \text{ mod } 2, (g_{0,1,1} \star_n g_{0,1,2},\dots, g_{k,r_k,1} \star_{q_k} g_{k,r_k,2})).
\end{equation}

In other words, we are given the following system of equations:

\begin{align*}
    z_1 + z_2 \equiv z_3 + z_4\mod 2,  \\
    g_{0,1,1} + g_{0,1,2} = g_{0,1,3} + g_{0,1,4}, \\ 
    \vdots \\
    g_{k,r_k,1} \star_{q_k} g_{k,r_k,2} = g_{k,r_k,3} \star_{q_k} g_{k,r_k,4}.
\end{align*}

Consider $q = (\phi(x_1),\phi(x_2),\phi(x_3),\phi(x_4))$. For $q$ to be a collision in $\mathrm{Dih}(G')$, we require the following system of equations:

\begin{align*}
    z_1 + z_2 \equiv z_3 + z_4\mod 2,  \\
    g_{0,1,1} \star_\alpha g_{0,1,2} = g_{0,1,3} \star_\alpha g_{0,1,4}, \\ 
    \vdots \\
    g_{k,r_k,1} \star_{q_k} g_{k,r_k,2} = g_{k,r_k,3} \star_{q_k} g_{k,r_k,4},
\end{align*}
which is implied by our given system. Therefore the result follows.
\end{proof}

\vspace{0.3cm}
We present an example of the number of collisions in another dihedral group $\mathbb{Z}_2 \ltimes \mathbb{Z}_n^2$. We show that the dihedral group $\mathbb{Z}_2 \ltimes \mathbb{Z}_n^2$ possesses exactly the same number of possible collisions as the dihedral group $\mathbb{Z}_2 \ltimes \mathbb{Z}_{n^2}$ if and only if $n$ is odd and there are more collisions in $\mathbb{Z}_2 \ltimes \mathbb{Z}_n^2$ otherwise. This result only provides some intuition on the the expected number of collisions depending on $j$, the number of elements of order $2$ within the particular abelian group.
\begin{lem} \label{eqnumcollisions}
    The number of possible collisions within the two groups are equal if and only if $n$ is odd and there are more collisions in $\mathbb{Z}_2 \ltimes \mathbb{Z}_n^2$ otherwise.
\end{lem}

To do so, we make use of the following useful results:

\begin{lem}\label{pairdiff}
    The number of pairs $(a_1,a_2),(b_1,b_2) \in \Z_n^2$ where $a_1-b_1 \equiv x \mod n$ and $a_2-b_2 \equiv y \mod n$ and the number of pairs $a_1n+a_2,b_1n+b_2 \in \Z_{n^2}$ where $(a_1n+a_2)-(b_1n+b_2) \equiv xn+y \mod n^2$ are both equal to $n^2$
\end{lem}
\begin{proof}
    For any given $(b_1,b_2)$ and $b_1n+b_2$, there exist only a single $(a_1,a_2)$ and $a_1n+a_2$ which satisfies the conditions respectively. Thus, there are $n^2$ such pairs for both cases.
\end{proof}

\begin{lem}\label{pairsumtwo}
    The number of pairs $(a_1,a_2),(b_1,b_2) \in \Z_n^2$ where $a_1+b_1 \equiv x \mod n$ and $a_2+b_2 \equiv y \mod n$ is as follows:
    \begin{itemize}
        \item when $n$ is odd, there are $\frac{n^2+1}{2}$ such pairs
        \item when $n$ is even and $y$ is odd or $x$ is odd, there are $\frac{n^2}{2}$ such pairs
        \item when $n$, $y$, and $x$ are all even, there are $\frac{n^2+4}{2}$ such pairs
    \end{itemize}
\end{lem}
\begin{proof}
    When $n$ is odd, each choice of $a_1$ is paired with a single possible $b_1$, with one such pair being $a_1=b_1$. Similarly, the same holds for $a_2$ and $b_2$. This gives $n^2$ as the number of pairs. However, we over-counted since swapping $a_1$ with $b_1$ and $a_2$ with $b_2$ does not yield a distinct pair. Thus, we divide by $2$ except for the single pair where $a_1=b_1$ and $a_2=b_2$ which we did not over-count to get $\frac{n^2-1}{2}+1=\frac{n^2+1}{2}$ pairs of $(a_1,a_2),(b_1,b_2)$.

    When $n$ is even but $y$ is odd, we choose pairs $(a_2,b_2)$ first which gives $n/2$ distinct pairs where $a_2$ is never equal to $b_2$. For each of these pairs, we can choose any choice of $a_1$ which forces $b_1$ without over-counting. Thus, we get $\frac{n^2}{2}$ pairs. Similar arguments hold for when $x$ is odd.

    When all of $n,y,$ and $x$ are even, we get that there are two pairs of identical number and $\frac{n-2}{2}$ pairs of different numbers which sum to $x$ and similar for $y$ in $\Z_n$. We choose $a_2,b_2$ first. If $a_2=b_2$, we have $\frac{n-2}{2}+2 = \frac{n+2}{2}$ choices for choosing $a_1,b_1$. If $a_2 \neq b_2$, we have that for any $a_1$, we uniquely determines $b_1$ without over-counting and thus there are $n$ choices. In total, we have $2*\frac{n+2}{2}+n*\frac{n-2}{2}=\frac{n^2+4}{2}$ pairs of $(a_1,a_2),(b_1,b_2)$.
\end{proof}

\begin{lem}\label{pairsumsquare}
    The number of pairs $a_1n+a_2,b_1n+b_2 \in \Z_{n^2}$ where $(a_1n+a_2)+(b_1n+b_2) \equiv xn+y \mod n^2$ is as follows:
    \begin{itemize}
        \item when $n$ is odd, there are $\frac{n^2+1}{2}$ such pairs
        \item when $n$ is even but $y$ is odd, there are $\frac{n^2}{2}$ such pairs
        \item when both $n$ and $y$ are even, there are $\frac{n^2+2}{2}$ such pairs
    \end{itemize}
\end{lem}
\begin{proof}
    When $n$ is odd, we have that each of $a_1n+a_2$ is paired with another $b_1n+b_2$, with exactly one being paired with itself. Thus, there are $\frac{n^2-1}{2}+1 = \frac{n^2+1}{2}$ pairs of $a_1n+a_2$ and $b_1n+b_2$. 

    When $n$ is even but $y$ is odd, we have that each of $a_1n+a_2$ is paired with another $b_1n+b_2$, necessarily distinct. Thus, there are $\frac{n^2}{2}$ pairs of $a_1n+a_2$ and $b_1n+b_2$. 

    When $n$ is even and $y$ is even, we have that each of $a_1n+a_2$ is paired with another $b_1n+b_2$, with exactly two being paired with themselves. Thus, there are $\frac{n^2-2}{2}+2 = \frac{n^2+2}{2}$ pairs of $a_1n+a_2$ and $b_1n+b_2$. 
\end{proof}

Combining these results over the casework where elements of our pairs may be a rotation or a flip yield the desired result, with the details of the proof in Appendix \ref{appendix eqnumcol}.

\section{Expected Size of Sum and Difference Sets}\label{expectation section}

In this section, we only consider the classical dihedral groups
\begin{equation}
    D_{2n} = \Z_2\ltimes \Z_n = \langle r, s \mid r^n, s^2, rsrs\rangle
\end{equation}
Throughout this section, we use $\S{m}$ to denote the set of subsets of size $m$ in $D_{2n}$.

The method of collision analysis will likely not be sufficient to prove that $\S{m}$ has more MSTD sets than MDTS sets for values of $m$ greater in order of magnitude than $\sqrt{n}$. The intuition for this comes from the fact that the sum and difference sets for $A\subseteq D_{2n}$ should very roughly have size on the same order of magnitude as $A^2$. Hence, one would expect to usually have $A+A = A-A = D_{2n}$ when $m$ is much greater than $\sqrt{n}$. The analysis for relative numbers of MSTD and MDTS sets in $\S{m}$ for these larger values of $m$ should therefore be based on counting the number of \textit{missed} sums and differences in $D_{2n}$, in direct analogy with the case of slow decay for the integers in \cite{hegarty2009almost}.

We take the first steps toward such an analysis by proving the following special case.

\expectationdifprime*

This follows from the following straightforward yet useful lemma, which reduces the problem of computing the probability of missing a sum or difference to an analogous problem in $\Z_n$.

\begin{lem}\label{lem:missing_dihedral_to_cyclic}
    Let $A$ be a subset in $\S{m}$ chosen uniformly at random. Then if $r^i$ is a rotation in $D_{2n}$, we have
    \begin{equation}\label{eq:rot_not_in_sum}
        \P[r^i\notin A+A] \ =\ \sum_{k=0}^m\ \begin{matrix*}[l]\ \ \P[\text{$A$ has $k$ flips}]\\ \cdot\ \P[i\notin S+S|\text{$S\subseteq \Z_n$, $|S|=m-k$}]\\ \cdot\ \P[i\notin S-S|\text{$S\subseteq \Z_n$, $|S|=k$}],\end{matrix*}
    \end{equation}
    and
    \begin{equation}\label{eq:rot_not_in_dif}
        \P[r^i\notin A-A] \ =\ \sum_{k=0}^m\ \begin{matrix*}[l]\ \ \P[\text{$A$ has $k$ flips}]\\ \cdot\ \P[i\notin S-S|\text{$S\subseteq \Z_n$, $|S|=m-k$}]\\ \cdot\ \P[i\notin S-S|\text{$S\subseteq \Z_n$, $|S|=k$}].\end{matrix*}
    \end{equation}
    If $r^is$ is a flip in $D_{2n}$, we have
    \begin{equation}\label{eq:flip_not_in_sum}
        \P[r^is\notin A+A] \ =\ \sum_{k=0}^m\ \begin{matrix*}[l]\ \ \P[\text{$A$ has $k$ flips}]\\ \cdot\ \P[i\notin S_1+S_2\land i\notin S_2-S_1|\text{$S_1,S_2\subseteq \Z_n$, $|S_1|=m-k$, $|S_2|=k$}],\end{matrix*}
    \end{equation}
    and
    \begin{equation}\label{eq:flip_not_in_dif}
        \P[r^is\notin A-A] \ =\ \sum_{k=0}^m\ \begin{matrix*}[l]\ \ \P[\text{$A$ has $k$ flips}]\\ \cdot\ \P[i\notin S_1+S_2|\text{$S_1,S_2\subseteq \Z_n$, $|S_1|=m-k$, $|S_2|=k$}].\end{matrix*}
    \end{equation}
\end{lem}

\begin{proof}
    Partition $A$ into its set of rotations $R$ and flips $F$, and suppose that $|F|=k$ and $|R|=m-k$. The rotation element $r^i$ can appear in $A+A$ precisely as $r^jr^\ell = r^{j+\ell}$ for $r^j,r^\ell\in R$ or as $(r^js)(r^\ell s) = r^{j-\ell}$ for $r^js,r^\ell s\in F$. Taking the probability of the negations of these events respectively give second and third probabilities appearing in the sum of Equation \eqref{eq:rot_not_in_sum}. Equation \eqref{eq:rot_not_in_dif} follows similarly.
    
    The flip element $r^is$ can appear in $A+A$ precisely as $r^j(r^\ell s) = r^{j+\ell}s$ or as $(r^\ell s)r^j = r^{\ell-j}s$ for $r^j\in R$, $r^\ell s\in F$, but unlike in the previous cases, these events are no longer independent so Equation \eqref{eq:flip_not_in_sum} cannot be broken up into a product of simpler probabilities. Finally, $r^is$ can appear in $A-A$ precisely $r^j(r^\ell s)^{-1} = r^{j+\ell}s = (r^{\ell}s)(r^j)^{-1}$, from which Equation \eqref{eq:flip_not_in_dif} follows.
\end{proof}

A number of these probabilities can be expressed explicitly in terms of $n$, $m$, and $k$. To prove Theorem \ref{thm:expectation_dif_prime}, we need to compute all probabilities appearing in Equation \eqref{eq:rot_not_in_dif} and Equation \eqref{eq:flip_not_in_dif}, but we will also compute the probabilities appearing in Equation \eqref{eq:rot_not_in_sum} for completeness.

\begin{lem}
    Let $S$ be a subset of $\Z_n$ of size $m-k$ chosen uniformly at random, and let $i$ be any element of $\Z_n$. Then
    \begin{equation}
        \P[i\notin S+S]\ =\ \begin{cases}\frac{2^{m-k}{\binom{\frac{n}{2}-1}{k}}}{{\binom{n}{m -k}}},\ \text{$n$ and $i$ both even,}\\
        \frac{2^{m-k}{\binom{\frac{n}{2}}{m-k}}}{{\binom{n}{m -k}}},\ \text{$n$ even and $i$ odd,}\\
        \frac{2^{m-k}{\binom{\frac{n-1}{2}}{m-k}}}{{\binom{n}{m -k}}},\ \text{$n$ odd.}
        \end{cases}
    \end{equation}
\end{lem}

\begin{proof}
    If $n$ and $i$ are both even, then there exist exactly $2$ elements of $\Z_n$ that give $i$ when added to themselves. The remaining elements of $\Z_n$ partition into pairs of distinct elements adding to $i$, and any $S$ such that $i\notin S+S$ is obtained by selecting $m-k$ of these $\frac{n}{2}-1$ pairs and one element from each pair, hence the result. The case where $n$ is even and $i$ is odd is identical except that all $n$ of the elements of $\Z_n$ are now partitioned into pairs as there are no elements that give $i$ when added to themselves. Finally, if $n$ is odd, then there is always a unique element that gives $i$ when added to itself, and the remaining elements partition into $\frac{n-1}{2}$ pairs, after which the same analysis can be applied.
\end{proof}

\begin{lem}\label{lem:cyclic_missing_dif}
    Let $S$ be a subset of $\Z_n$ of size $k$ chosen uniformly at random, and let $i$ be any nonzero element of $\Z_n$ of order $n/d$. Then
    \begin{equation}\label{eq:cyclic_missing_dif}
        \P[i\notin S-S]\ =\ \frac{1}{{\binom{n}{k}}}\sum_{\substack{(k_1,\ldots,k_d)\in \Z_{\ge 0}^d \\ k_1+\cdots+k_d=k}} \prod_{t=1}^d g\left(\frac{n}{d},k_t\right),
    \end{equation}
    where
    \begin{equation}
        g\left(\frac{n}{d},k_t\right)\ =\ \begin{cases} 1,\ k_t=0,\\ \frac{\frac{n}{d}{\binom{\frac{n}{d}-1-k_t}{k_t-1}}}{k_t},\ k_t>0.\end{cases}
    \end{equation}
\end{lem}
\begin{rmk}
    When $n$ is prime, we must have $d=1$ always and Equation \eqref{eq:cyclic_missing_dif} simplifies significantly to
    \begin{equation}
        \P[i\notin S-S]\ =\ \frac{g(n,k)}{{\binom{n}{k}}}\ =\ \frac{1}{{\binom{n}{k}}}\cdot\begin{cases}1,\ k= 0,\\ \frac{n{\binom{n-1-k}{k-1}}}{k},\ k>0.\end{cases}
    \end{equation}
    This is the reason why we restrict to such $n$ in Theorem \ref{thm:expectation_dif_prime}.
\end{rmk}

\begin{proof}
    Partition $\Z_n$ into the $d$ additive cosets of the subgroup $i\Z_n$. Each of these cosets has size $n/d$ and has elements that can be cyclically ordered such that the difference between any element and its predecessor is $i$. Choosing a set $S$ of size $k$ such that $i\notin S-S$ is then equivalent to partitioning $k$ into $k_1+\cdots+k_d$ with each $k_t\ge 0$, and choosing $k_t$ non-adjacent elements from the $t^{\text{th}}$ coset to include in $S$. This is precisely what is counted by Equation \eqref{eq:cyclic_missing_dif} if $g(n/d,k_t)$ is equal to the number of ways to choose $k_t$ non-adjacent elements from a cyclically ordered set of size $n/d$.
    
    Indeed, this is clear for $k_t=0$, so assume $k_t>0$. There are $n/d$ choices for the first element to be included, and each selection of $k_t$ elements may be made in $k_t$ different ways by designating different elements to be this first choice. Once the first element has been selected, the number of ways to choose the remaining elements is equal to the number of ways to choose $k_t-1$ non-adjacent elements from a linearly ordered set of size $n/d-1$ without choosing the extremal elements. This is the classic stars and bars problem, which yields precisely the binomial coefficient in the definition of $g(n/d,k_t)$.
\end{proof}

\begin{lem}\label{lem:cyclic_two_set_missing_sum}
    Let $S_1$ and $S_2$ be subsets of $\Z_m$ of sizes $m-k$ and $k$, respectively, chosen uniformly at random, and let $i$ be any element of $\Z_n$. Then
    \begin{equation}
        \P[i\notin S_1+S_2]\ =\ \frac{{\binom{n-k}{m-k}}}{{\binom{n}{m -k}}}.
    \end{equation}
\end{lem}

\begin{proof}
    For any fixed choice of $S_2$, we have $\binom{n}{m-k}$ total choices for $S_1$. For each of these $S_1$, we have $i\notin S_1+S_2$ if and only if $i-j\notin S_1$ for all $j\in S_2$. Because $|S_2|=k$, this leaves $\binom{n-k}{m-k}$ choices for $S_1$ such that $i\notin S_1+S_2$.
\end{proof}

We now have all the necessary parts to prove Theorem \ref{thm:expectation_dif_prime}. Simple combinatorics yields
\begin{equation}
    \P[\text{$A$ has $k$ flips}]\ =\ \frac{{\binom{n}{k}}{\binom{n}{m-k}}}{{\binom{2n}{m}}}.
\end{equation}
Combining this with Lemmas \ref{lem:missing_dihedral_to_cyclic}, \ref{lem:cyclic_missing_dif}, and \ref{lem:cyclic_two_set_missing_sum} yields the result after some simplification. See Appendix \ref{appendix expectation} for complete details.

By plotting $\E[|A-A|]$ against $m$ for certain large values of $n$, we can see evidence for our intuition from the beginning of this section that $\E[|A-A|]$ quickly becomes close to $2n$ for $m$ on the order of magnitude of $\sqrt{n}$ (Figure \ref{fig:ex_A-A_vs_m}). Numerical evidence from many large primes $n$ suggests that the value of $m$ such that $\E[|A-A|]=n$ is about $1.3875\sqrt{n}$.\footnote{Some code written for this project can be found at \url{https://github.com/ZeusDM/MSTD-experiments}.}

\begin{figure}[htbp]
    \centering
    \includegraphics[width = 0.8\textwidth]{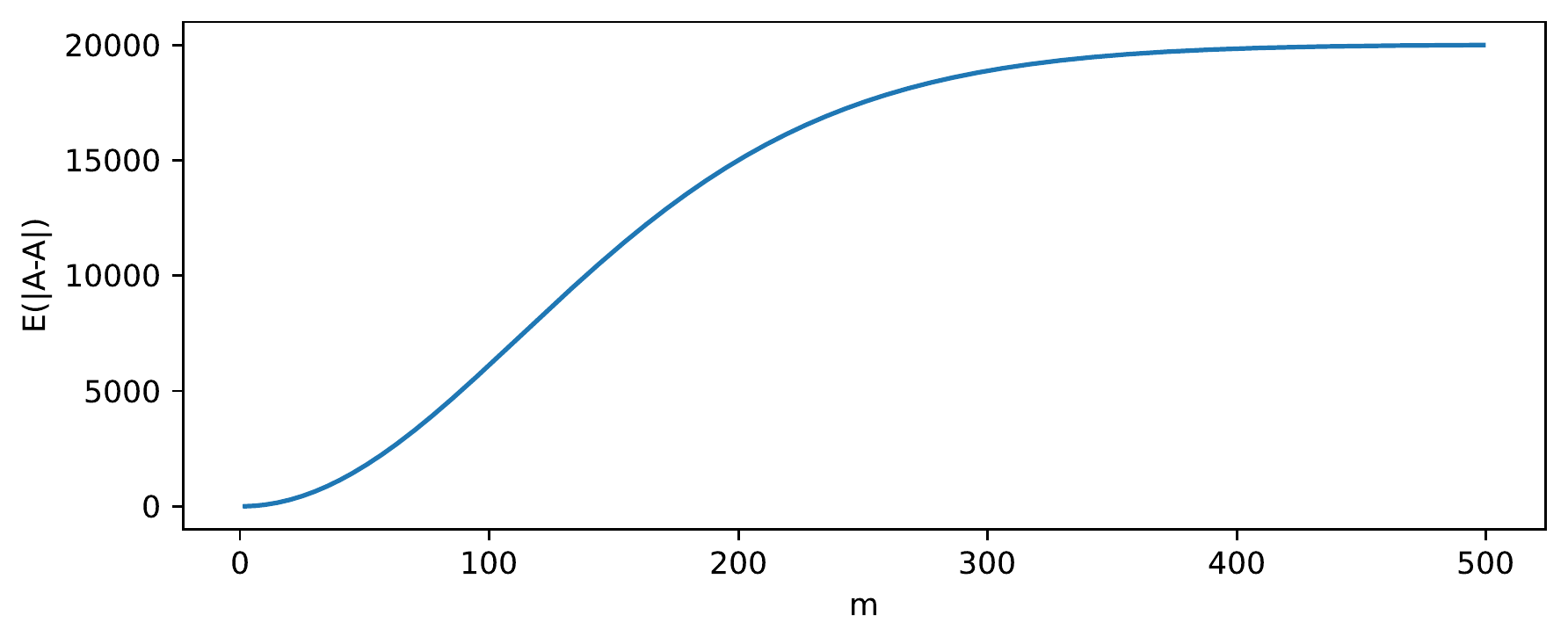}
    \caption{A plot of $\E[|A-A|]$ versus $m$ for $A\in \S{m}$ with $n=10007$. Note that $\E[|A-A|]$ rapidly approaches $2n = 20014$ as $m$ approaches a small multiple of $\sqrt{n}\approx 100$.}
    \label{fig:ex_A-A_vs_m}
\end{figure}
\section{Future Work}\label{future work}

An immediate future direction of research is to extend the bounds on $m$ to show Conjecture \ref{conj:SDm-mmstdtmdts} for all $m$ and $n$. However, it is unlikely that the methods we have used in this paper will be useful for larger values of $m$. This is because our approach showed that for values of $m$ that we considered, the majority of subsets of the generalized dihedral group with size $m$ were MSTD. But this is simply not the case for larger $m$: numerical evidence shows that for any $m\gg \sqrt{n}$, the vast majority of the sets are balanced. A new approach is required to show that out of the sets that are not balanced, more are MSTD.

To this end, it would also be productive to more carefully analyze how many elements of the group $D$ are not in $A+A$ (or not in $A-A$) for $|A|$ closer to $n$. This follows the approach of \cite{hegarty2009almost} for the ``slow decay" case. Such an analysis could be used to find explicit formulas for the expectation values of $|A+A|$ and $|A-A|$. In this paper we found a formula for $\E[|A-A|]$, but only when $n$ is prime. Further, to use these results, we would also need to bound the variance of $| A + A |$ and $| A - A |$. Depending on the results, this could be enough to prove Conjecture \ref{conj:gendih-mmstdtmdts} if we can deduce an upper bound $M_\ell$ on the number of MDTS sets and a lower bound $M_s$ on the number of MSTD sets such that $M_\ell < M_s$.

Yet another possible approach to prove Conjecture \ref{conj:gendih-mmstdtmdts} is to construct an injective map from MDTS sets to MSTD sets in the group. Such an approach has proven to be difficult, but has the potential advantage of working for both large and small values of $m$.


\appendix
\section{Proof of Lemmas \ref{S2n2} and \ref{S2n3}}\label{base cases}

\stwontwo*

\begin{proof}
There are only 3 possible cases to consider for $A$.  

\begin{itemize}
    \item $A$ contains two flip elements: in this case, adding and subtracting flips is an identical operation as each flip element has order 2.  Thus, $A$ will necessarily be sum-difference balanced.  
    \item $A$ contains one flip and one rotation element: let $r^i,r^js\in A$.  Here, $A+A=\{1,r^{2i}, r^{i+j}s, r^{j-i}s\}$.  However, $A-A=\{1,r^{i+j}s\}$.  When $r^{2i}\ne e$ or $r^{j-i}s\ne r^{i+j}s$, $A$ will be MSTD.  In the case where both of these are actually equalities ($A+A=\{e,r^{i+j}s\}$), $A$ will simply be balanced.
    \item $A$ contains two rotation elements: let $r^a,r^b\in A$.  Here, $A+A=\{r^{2i},r^{i+j},r^{2j}\}$.  However, $A-A=\{1,r^{i-j},r^{j-i}\}$.  Note that $i\ne j$.  Both the sum set and the difference set have 3 elements except in one special case.  Suppose that $r^{2i}=r^{2j}$.  Then, we have that $r^ir^{-j}=r^jr^{-i}\implies r^{i-j}=r^{j-i}$.  Thus, when $A$ contains two rotation elements only, $A$ is always balanced.
\end{itemize}
Therefore, $A$ has strictly more MSTD sets in $\mathcal{S}_{2n,2}$ than MDTS sets.
\end{proof}

\stwonthree*

\begin{proof}
There are 4 possible cases to consider for $A$.

\begin{itemize}
    \item $A$ contains three flip elements: just as in the $m=2$ case, since addition and subtraction of two flips are equivalent, $A$ is balanced.
    \item $A$ contains two flip elements and one rotation element: let $A=\{r^i,r^js,r^ks\}$ for $j\neq k$. Then the sumset is
    \begin{equation}
        A+A = \{r^{2i},r^{i+j}s,r^{i+k}s,r^{j-k},r^{k-j},1,r^{j-i}s,r^{k-i}s\},
    \end{equation}
    while the difference set is
    \begin{equation}
        A-A = \{r^{i+j}s,r^{i+k}s,r^{j-k},r^{k-j},1\},
    \end{equation}
    so $A-A\subseteq A+A$. Moreover, $A$ is MSTD precisely when at least one of $r^{j-i}s$ and $r^{k-i}s$ is distinct from each of the flips in $A+A$ that lie in $A-A$. We enumerate all the ways that this could fail to occur:
    
    We have $r^{i+j}s=r^{j-i}s$ and $r^{i+k}s = r^{k-i}s$ if and only if $i = -i\pmod{n}$. This happens if and only if $i=0$ or if $n$ is even and $i = n/2$, and in each of these cases we have $\binom{n}{2}=\frac{1}{2}(n^2-n)$ choices for $j$ and $k$. On the other hand, we have $r^{i+j}s=r^{k-i}s$ and $r^{i+k}s=r^{j-i}s$ if and only if $2i=k-j=j-k\pmod{n}$, which implies that $2j=2k\pmod{n}$. We also require $2i=1\pmod{n}$; because $j\neq k$, these equations can occur if and only if $n$ is divisible by $4$, $i = n/4$ or $3n/4$, and $k-j = n/2\pmod{n}$. We therefore have $2$ choices for $i$, and $\frac{1}{2}n$ choices for $j$ and one corresponding choice of $k$ for each $i$. Finally, note that we cannot have $r^{j-i}s=r^{k-i}s$ because $k\neq i$, so this completes the enumeration of all such sets $A$ that fail to be MSTD.
    
    There are $n\binom{n}{2} = \frac{1}{2}(n^3-n^2)$ total sets $A$ containing exactly $2$ flips. Subtracting the balanced sets that we just enumerated, we obtain the following numbers of MSTD sets:
    \begin{align}\label{eq:m3k2mstd}
        \frac{1}{2}(n^3-n^2)\ &\text{if $n=1\pmod{2}$,}\nonumber\\
        \frac{1}{2}(n^3-n^2)-n(n-1)\ &\text{if $n=2\pmod{4}$,}\nonumber\\
        \frac{1}{2}(n^3-n^2)-n(n-1)-n\ &\text{if $n=0\pmod{4}$.}
    \end{align}
    \item $A$ contains one flip element and two rotation element: let $A=\{r^i, r^j, r^ks\}$ for $i\neq j$. Then the sumset is     
    \begin{equation}
        A+A = \{r^{i+j},r^{2i},r^{2j},r^{i+k}s,r^{j+k}s,r^{k-i}s,r^{k-j}s,1\},
    \end{equation}
    while the difference set is
    \begin{equation}
        A-A = \{r^{i-j},r^{j-i},r^{i+k}s,r^{j+k}s,1\},
    \end{equation}
    We show that $A$ cannot be MDTS by separately comparing the flips and the rotations in $A+A$ and $A-A$. The flips in $A-A$ are contained in $A+A$, so this comparison is trivial. The rotations in $A-A$ are $r^{i-j}, r^{j-i},$ and $1$; we will show that $A+A$ has at least as many rotations. Consider the rotations in $A+A$: note that $r^{i+j}$ is different from $r^{2i}$ and from $r^{2j}$ since $i\neq j$, so there are at least $2$ distinct rotations in $A+A$. If $r^{2i}\neq r^{2j}$, then in fact $A+A$ has at least three distinct rotations, and we have $|A+A|\geq |A-A|$. On the other hand, if $r^{2i}=r^{2j}$, then $r^{i-j} = r^{j-i}$, so there are only at most $2$ distinct rotations in $A-A$, so again $|A+A|\geq |A-A|$.
    
    \item $A$ contains three rotation elements: for this case, there are $\binom{n}{3}$ such subsets.  We compare directly to Equation \eqref{eq:m3k2mstd}, the case that yields the least possible number of MSTD sets in $\mathcal{S}_{2n,3}$.  Here we have 
    \begin{equation}
        \frac{1}{2}(n^3-n^2)-n(n-1)-n>\frac{1}{6}n(n-1)(n-2)
    \end{equation}
    or equivalently,
    \begin{equation}
        n^3-3n^2-n>0,
    \end{equation}
    which is true if $n\ge 4$. When $n=3$, since $n=1\mod 2$, we compare $\binom{n}{3}$ to $\frac{1}{2}(n^3-n^2)$ instead, and the result is verified once again. 
    
\end{itemize}
Therefore, even if all of the subsets with three rotation elements are MDTS, which is never true to begin with, we still have strictly more MSTD subsets of size 3 than MDTS subsets of size 3.
\end{proof}


\section{Proof of Lemma \ref{eqnumcollisions}}\label{appendix eqnumcol}
\setlist[itemize]{topsep=1pt}
    For any given $(x,y,F) \in \Z_2 \ltimes \Z_n^2$, we calculate the number of pairs $a,b \in \Z_2 \ltimes \Z_n^2$ where $a-b = (x,y,F)$. Similarly, we also calculate the number of pairs $a,b \in \Z_2 \ltimes \Z_n^2$ where $a+b = (x,y,F)$. Note that these can be counted through counting the number of pairs in $\Z_n^2$ which results in $(x,y)$ as follows:

Case 1, $F=0$ \begin{itemize}
    \item Case 1.1, $a-b = (x,y,0)$ \begin{itemize}
        \item Case 1.1.1 $a$ and $b$ are not flips. 
        
        We get that this is equivalent to counting pairs $(a_1,a_2),(b_1,b_2) \in \Z_n^2$ where $a_1-b_1 \equiv x \mod n$ and $a_2-b_2 \equiv y \mod n$ 
        
        \item Case 1.1.2 $a$ and $b$ are both flips. 
        
        We get that this is equivalent to counting pairs $(a_1,a_2),(b_1,b_2) \in \Z_n^2$ where $a_1-b_1 \equiv x \mod n$ and $a_2-b_2 \equiv y \mod n$ 
        \end{itemize}
    
    \item Case 1.2, $a+b = (x,y,0)$ \begin{itemize}
        \item Case 1.2.1 $a$ and $b$ are not flips. 
        
        We get that this is equivalent to counting pairs $(a_1,a_2),(b_1,b_2) \in \Z_n^2$ where $a_1+b_1 \equiv x \mod n$ and $a_2+b_2 \equiv y \mod n$ 
        
        \item Case 1.2.2 $a$ and $b$ are both flips. 
        
        We get that this is equivalent to counting pairs $(a_1,a_2),(b_1,b_2) \in \Z_n^2$ where $a_1-b_1 \equiv x \mod n$ and $a_2-b_2 \equiv y \mod n$ 
        \end{itemize}
\end{itemize}

Case 2, $F = 1$ \begin{itemize}
    \item Case 2.1 $a-b = (x,y,1)$ \begin{itemize}
        \item Case 2.1.1 $a$ is a flip. 
        
        We get that this is equivalent to counting pairs $(a_1,a_2),(b_1,b_2) \in \Z_n^2$ where $a_1+b_1 \equiv x \mod n$ and $a_2+b_2 \equiv y \mod n$ 
        
        \item Case 2.1.2 $b$ is a flip. 
        
        We get that this is equivalent to counting pairs $(a_1,a_2),(b_1,b_2) \in \Z_n^2$ where $a_1+b_1 \equiv x \mod n$ and $a_2+b_2 \equiv y \mod n$ 
        \end{itemize}
        
    \item Case 2.2 $a+b = (x,y,1)$ \begin{itemize}
        \item Case 2.2.1 $a$ is a flip. 
        
        We get that this is equivalent to counting pairs $(a_1,a_2),(b_1,b_2) \in \Z_n^2$ where $a_1-b_1 \equiv x \mod n$ and $a_2-b_2 \equiv y \mod n$ 
        
        \item Case 2.2.2 $b$ is a flip. 
        
        We get that this is equivalent to counting pairs $(a_1,a_2),(b_1,b_2) \in \Z_n^2$ where $a_1+b_1 \equiv x \mod n$ and $a_2+b_2 \equiv y \mod n$ 
        \end{itemize}
\end{itemize}

We get a similar result with $a,b \in \Z_2 \ltimes \Z_{n^2}$ where $a-b = (xn+y,F)$ or $a+b = (xn+y,F)$ as follows:

Case 1, $F=0$ \begin{itemize}
    \item Case 1.1, $a-b = (xn+y,0)$ \begin{itemize}
        \item Case 1.1.1 $a$ and $b$ are not flips. 
        
        We get that this is equivalent to counting pairs $a_1n+a_2,b_1n+b_2 \in \Z_{n^2}$ where $(a_1n+a_2)-(b_1n+b_2) \equiv xn+y \mod n^2$
        
        \item Case 1.1.2 $a$ and $b$ are both flips. 
        
        We get that this is equivalent to counting pairs $a_1n+a_2,b_1n+b_2 \in \Z_{n^2}$ where $(a_1n+a_2)-(b_1n+b_2) \equiv xn+y \mod n^2$
        \end{itemize}
    
    \item Case 1.2, $a+b = (xn+y,0)$ \begin{itemize}
        \item Case 1.2.1 $a$ and $b$ are not flips. 
        
        We get that this is equivalent to counting pairs $a_1n+a_2,b_1n+b_2 \in \Z_{n^2}$ where $(a_1n+a_2)+(b_1n+b_2) \equiv xn+y \mod n^2$
        
        \item Case 1.2.2 $a$ and $b$ are both flips. 
        
        We get that this is equivalent to counting pairs $a_1n+a_2,b_1n+b_2 \in \Z_{n^2}$ where $(a_1n+a_2)-(b_1n+b_2) \equiv xn+y \mod n^2$
        \end{itemize}
\end{itemize}

Case 2, $F = 1$ \begin{itemize}
    \item Case 2.1 $a-b = (xn+y,1)$ \begin{itemize}
        \item Case 2.1.1 $a$ is a flip. 
        
        We get that this is equivalent to counting pairs $a_1n+a_2,b_1n+b_2 \in \Z_{n^2}$ where $(a_1n+a_2)+(b_1n+b_2) \equiv xn+y \mod n^2$
        
        \item Case 2.1.2 $b$ is a flip. 
        
        We get that this is equivalent to counting pairs $a_1n+a_2,b_1n+b_2 \in \Z_{n^2}$ where $(a_1n+a_2)+(b_1n+b_2) \equiv xn+y \mod n^2$
        \end{itemize}
        
    \item Case 2.2 $a+b = (xn+y,1)$ \begin{itemize}
        \item Case 2.2.1 $a$ is a flip. 
        
        We get that this is equivalent to counting pairs $a_1n+a_2,b_1n+b_2 \in \Z_{n^2}$ where $(a_1n+a_2)-(b_1n+b_2) \equiv xn+y \mod n^2$
        
        \item Case 2.2.2 $b$ is a flip. 
        
        We get that this is equivalent to counting pairs $a_1n+a_2,b_1n+b_2 \in \Z_{n^2}$ where $(a_1n+a_2)+(b_1n+b_2) \equiv xn+y \mod n^2$
        \end{itemize}
\end{itemize}
Note that the signs are the same between corresponding cases in $\Z_2 \ltimes \Z_{n^2}$ and $\Z_2 \ltimes \Z_n^2$. And thus there are only four distinct cases we need to count, and it suffices to compare the number of 4-tuple collisions over all $(x,y) \in \Z_n^2$ with the number of 4-tuple collisions over all $xn+y \in \Z_{n^2}$

We will prove the main result by counting the collisions of each of the 3 types. When $n$ is odd, it is simple to see from Lemma \ref{pairdiff}, Lemma \ref{pairsumtwo}, and Lemma \ref{pairsumsquare} that the number of collisions for each of the $(x,y)$ in $\Z_n^2$ is the same for the corresponding $xn+y \in \Z_{n^2}$ and thus there are an equal total number of collisions. We now assume that $n$ is even. \begin{itemize}[topsep=0cm]
        \item $a+b=c+d$ \\
        Suppose $y$ is odd, from Lemma \ref{pairsumtwo} we have $(\frac{n^2}{2})^2$ pairs for each $(x,y)$ for a total of $\frac{n^5}{4}$ collisions. From Lemma \ref{pairsumsquare}, we also get $(\frac{n^2}{2})^2$ pairs for each $xn+y$ for an equal total $\frac{n^5}{4}$ collisions.

        If $y$ is even, we have a half of $(x,y)$ yielding $(\frac{n^2}{2})^2$ pairs and another half yielding $(\frac{n^2+4}{2})^2$ pairs for a total of $\frac{n}{2}*\frac{n^4+n^4+8n^2+16}{4}=\frac{n^5+4n^3+8n}{4}$. Meanwhile in $\Z_{n^2}$ we have each $xn+y$ yielding $(\frac{n^2+2}{2})^2$ for an equal total of $\frac{n^5+4n^3+4n}{4}$. Thus, there are strictly more collisions in $\Z_n^2$ in this case.
        
        \item $a+b=c-d$ \\
        Suppose $y$ is odd, we have $\frac{n^2}{2}$ choices for $a+b$ and $n^2$ choices for $c-d$ for a total of $\frac{n^5}{2}$. We also get the same results for $\Z_{n^2}$.

        If $y$ is even, we have half of $(x,y)$ yielding $\frac{n^2}{2}*n^2$ pairs and another half yielding $\frac{n^2+4}{2}*n^2$ pairs for a total of $\frac{n}{2}*\frac{n^4+n^4+4n^2}{2}=\frac{n^5+2n^3}{2}$. For $\Z_{n^2}$, we have each $(x,y)$ yielding $\frac{n^2+2}{2}*n^2$ collisions for a total of $\frac{n^5+2n^3}{2}$

        \item $a-b=c-d$ \\
        Both groups have the same number of collisions for each $(x,y)$ and the corresponding $xn+y$ at $n^4$. Thus, the total number of collisions in this case is $n^5$. 
    \end{itemize}
    Thus, if we look over all the cases where $n$ is even, we found that either there are strictly more collisions in $\Z_n^2$  or there are equal amount of collisions. Thus, we can conclude that there are strictly more collisions in $\mathbb{Z}_2 \ltimes \mathbb{Z}_n^2$ when $n$ is even.


\section{Proof of Theorem \ref{thm:expectation_dif_prime}}\label{appendix expectation}
    We have
    \begin{align}
        \E[|A-A|]\ &=\ |D_{2n}| - \sum_{g\in D_{2n}} \P[g\notin A-A]\nonumber\\
        &=\ 2n - \P[1\notin A-A] - \left[\sum_{i=1}^{n-1} \P[r^i\notin A-A]\right] - \left[\sum_{i=0}^{n-1} \P[r^is\notin A-A]\right]\nonumber\\
        &=\ 2n - 0 - \left[\sum_{i=1}^{n-1} \sum_{k=0}^m\ \begin{matrix*}[l]\ \ \P[\text{$A$ has $k$ flips}]\\ \cdot\ \P[i\notin S-S|\text{$S\subseteq \Z_n$, $|S|=m-k$}]\\ \cdot\ \P[i\notin S-S|\text{$S\subseteq \Z_n$, $|S|=k$}]\end{matrix*}\right] \nonumber\\
        &\quad\quad\quad\quad\ \ - \left[\sum_{i=0}^{n-1}\sum_{k=0}^m\ \begin{matrix*}[l]\ \ \P[\text{$A$ has $k$ flips}]\\ \cdot\ \P[i\notin S_1+S_2|\text{$S_1,S_2\subseteq \Z_n$, $|S_1|=m-k$, $|S_2|=k$}]\end{matrix*}\right]\nonumber\\
        &=\ 2n - \left[\sum_{i=1}^{n-1}\sum_{k=0}^m \frac{{\binom{n}{k}}{\binom{n}{m-k}}}{{\binom{2n}{m}}}\cdot \frac{g(n,k)}{{\binom{n}{k}}}\cdot \frac{g(n,m-k)}{{\binom{n}{m -k}}}\right]\nonumber\\
        &\quad\quad\ \ \ -\left[\sum_{i=0}^{n-1}\sum_{k=0}^m\frac{{}{\binom{n}{m-k}}}{{\binom{2n}{m}}}\cdot \frac{{\binom{n-k}{m-k}}}{{\binom{n}{m -k}}}\right]\nonumber\\
        &=\ 2n - \left[(n-1)\sum_{k=0}^m \frac{g(n,k)\cdot g(n,m-k)}{{\binom{2n}{m}}}\right] - \left[n\sum_{k=0}^m \frac{{\binom{n}{k}}{\binom{n-k}{m-k}}}{{\binom{2n}{m}}}\right] \nonumber\\
        &=\ 2n - \left[2(n-1)\frac{g(n,0)\cdot g(n,m)}{{\binom{2n}{m}}} +(n-1)\sum_{k=1}^{m-1}\frac{\frac{n{\binom{n-1-k}{k-1}}}{k}\cdot \frac{n{\binom{n-1-m+k}{m-k-1}}}{m-k}}{{\binom{2n}{m}}}\right]\nonumber\\
        &\quad\quad\ \ \ -\left[\frac{n}{{\binom{2n}{m}}} \sum_{k=0}^m {\binom{n}{k}}{\binom{n-k}{m-k}}\right] \nonumber\\
        &=\ 2n-\frac{nm2^m{\binom{n}{m}}+2n(n-1){\binom{n-m-1}{m-1}}}{m{\binom{2n}{m}}} -\frac{n^2(n-1)}{{\binom{2n}{m}}}\sum_{k=1}^{m-1}\frac{{\binom{n+k-m-1}{m-k-1}}{\binom{n-k-1}{k-1}}}{k(m-k)},
    \end{align}
    by the combinatorial identity
    \begin{equation}
        \sum_{k=0}^m {\binom{n}{k}}{\binom{n-k}{m-k}}\ =\ 2^m {\binom{n}{m}}.
    \end{equation}


\printbibliography

\end{document}